\documentclass[a4paper,10pt]{article}

\usepackage{amsmath,amssymb,amsthm,amscd}
\usepackage[mathscr]{eucal}
\usepackage[all]{xy}
\usepackage{graphicx}

\makeatletter
    
    \@addtoreset{equation}{section}
  \makeatother

\newcommand{\plim}{\varprojlim}
\newcommand{\mcal}{\mathcal}
\newcommand{\mbf}{\mathbf}

\newcommand{\mbb}{\mathbb}
\newcommand{\mrm}{\mathrm}
\newcommand{\vphi}{\varphi}

\newcommand{\aet}{\mathrm{\acute{e}t}}
\newcommand{\cO}{\mathcal{O}}

\newtheorem{theorem}{Theorem}
\newtheorem{corollary}[theorem]{Corollary}
\newtheorem{lemma}[theorem]{Lemma}
\newtheorem{proposition}[theorem]{Proposition}

\theoremstyle{definition}

\newtheorem{remark}[theorem]{Remark}
\newtheorem{claim}{Claim}

\newtheorem*{question}{Question}

\newtheorem*{acknowledgments}{Acknowledgments}

\title{Bounds on torsion of CM abelian varieties 
over a $p$-adic field with values in a field of $p$-power roots} 

\author{Yoshiyasu Ozeki\footnote{
Department of Mathematics and Physics, Faculty of Science, Kanagawa University,
  2946 Tsuchiya, Hiratsuka-shi, Kanagawa 259--1293, JAPAN
\endgraf
e-mail: {\tt ozeki@kanagawa-u.ac.jp} 
\endgraf
This work is supported by JSPS KAKENHI Grant Number JP19K03433.}
}

\begin{document}
\maketitle

\begin{abstract}
Let $p$ be a prime number 
and $M$ the extension field of 
a $p$-adic field $K$ obtained by adjoining all $p$-power roots of all elements of $K$.
In this paper, 
we show that there exists a  constant $C$, 
depending only on  $K$ and an integer $g>0$, 
which satisfies the following property:
If $A_{/K}$ is a $g$-dimensional CM abelian variety,  
then the order of the torsion subgroup of $A(M)$
 is bounded by $C$.
\end{abstract}



\section{Introduction}

Let $p$ be a prime number.
Let $K$ be a number field (= a finite extension of $\mbb{Q}$) 
or a $p$-adic field (= a finite extension of $\mbb{Q}_p$).
Let $A$ be an abelian variety defined over $K$ of dimension $g$.
It follows from the Mordell-Weil theorem and the main theorem of \cite{Mat} that 
the torsion subgroup $A(K)_{\mrm{tors}}$ of $A(K)$ is finite.
The following question for $A(K)_{\mrm{tors}}$ 
is quite natural and have been studied for a long time:
\begin{question}
What can be said about the size of the order of $A(K)_{\mrm{tors}}$? 
\end{question}
\noindent
If $K$ is a number field of degree $d$ and $A$ is an elliptic curve (i.e., $g=1$), 
it is really surprising that there exists a constant $B(d)$, depending only on the degree $d$,
such that $\sharp A(K)_{\mrm{tors}}<B(d)$.
The explicit formula of  such a constant  $B(d)$ is given by 
Merel, Oesterl\'e and Parent
(cf.\ \cite{Me}, \cite{Pa}).  
The amazing point here is that the constant $B(d)$ is uniform 
in the sense that it
depends not on the number field $K$ but on the degree $d=[K:\mbb{Q}]$.
Such uniform boundedness results are not known 
for abelian varieties of dimension greater than one.
Next we consider the case where $K$ is a $p$-adic field.
As remarked by Cassels,
the "uniform boundedness theorem" for $p$-adic  base fields 
would be false (cf.\ Lemma 17.1 and p.264 of \cite{Ca}). 
For abelian varieties $A$ over $K$ with anisotropic reduction, 
Clark and Xarles \cite{CX} give an upper bound 
of the order of $A(K)_{\mrm{tors}}$ in terms of $g,p$ and some numerical invariants of 
$K$.
This includes the case in which $A$ has potentially good reduction, 
and in this case the existence of a bound can be found in some literatures
(cf.\ \cite{Si2}, \cite{Si3}).

We are interested in  the order of $A(L)_{\mrm{tors}}$ 
for certain algebraic extensions $L$ of $K$ {\it of infinite degree}.
Now we suppose that  $K$ is a $p$-adic field. 
There are not so many known $L$ so that $A(L)_{\mrm{tors}}$ is finite.
Imai \cite{Im} showed  that $A(L)_{\mrm{tors}}$  is finite 
if $A$ has potential good reduction and $L=K(\mu_{p^{\infty}})$,
where $\mu_{p^{\infty}}$ is the set  of $p$-power root of unity. 
The author \cite{Oz1} showed that Imai's finiteness result holds 
even if we replace $L=K(\mu_{p^{\infty}})$ with $L=Kk_{\pi}$, where 
$k$ is a $p$-adic field and $k_{\pi}$ is the Lubin-Tate extension
of $k$ associated with a certain uniformizer $\pi$ of $k$.
The result \cite{KT} of Kubo and  Taguchi is also interesting. 
They showed  that 
the torsion subgroup of $A(K(\sqrt[p^{\infty}]{K}))$ is finite, where
$A$ is an abelian variety over  $K$ with potential good reduction and 
$K(\sqrt[p^{\infty}]{K})$ is the extension field of $K$ obtained 
by adjoining all $p$-power roots of all elements of $K$. 
Our main theorem is motivated by the result of Kubo and Taguchi.
The goal of this paper is to show that, 
under the assumption that  $A$ has complex multiplication, 
the order of $A(K(\sqrt[p^{\infty}]{K}))_{\mrm{tors}}$ is "uniformly" bounded.

\begin{theorem}
\label{Main}
There exists a  constant $C(K,g)$, 
depending only on a $p$-adic field $K$ and an integer $g>0$, 
which satisfies the following property:
If $A$ is a $g$-dimensional abelian variety over $K$ with complex multiplication,  
then we have 
$$
\sharp A\left(K(\sqrt[p^{\infty}]{K})\right)_{\mrm{tors}}< C(K,g).
$$
\end{theorem}

The theorem above gives a  global result: 
For any integer $d>0$, 
we denote by $\mbb{Q}_{\le d}$ the composite 
of all number fields of degree $\le d$.
If we fix an embedding $\overline{\mbb{Q}}\hookrightarrow \overline{\mbb{Q}}_p$,
then $\mbb{Q}_{\le d}$ is embedded into the composite field
of all $p$-adic fields of degree $\le d$, 
which is a finite extension of $\mbb{Q}_p$. 
If we denote by  $\mbb{Q}_{\le d,p}$
the extension field of 
$\mbb{Q}_{\le d}$ obtained by adjoining 
all $p$-power roots of all elements of $\mbb{Q}_{\le d}$, 
then the following is an immediate consequence of  our main theorem. 
\begin{corollary}
There exists a  constant $C(d,g,p)$,  
depending only on positive integers $d,g$ and a prime number  $p$, 
which satisfies the following property:
If $A$ is a $g$-dimensional abelian variety over 
$\mbb{Q}_{\le d}$ with complex multiplication,  
then we have 
$$
\sharp A(\mbb{Q}_{\le d,p})_{\mrm{tors}}< C(d,g,p).
$$
\end{corollary}

\vspace{5mm}
\noindent
{\bf Notation :}
Throughout this paper, a $p$-adic field means a finite extension of $\mbb{Q}_p$
in a fixed algebraic closure $\overline{\mbb{Q}}_p$ of $\mbb{Q}_p$.
If $F$ is an algebraic extension of $\mbb{Q}_p$, 
we  denote by $\cO_F$ and $\mbb{F}_{F}$ the ring of integers of $F$ and the residue field of $F$,
respectively.  
We denote by $G_F$  the absolute Galois group of $F$ 
and also denote by $\Gamma_F$ the set of $\mbb{Q}_p$-algebra 
embeddings of $F$ into $\overline{\mbb{Q}}_p$.
We put $d_F=[F:\mbb{Q}_p]$. 
For an algebraic extension $F'/F$,
we denote by $e_{F'/F}$ and $f_{F'/F}$ 
the ramification index of $F'/F$ and 
the extension degree of the residue field extension of  $F'/F$,
respectively.
We set $e_F:=e_{F/\mbb{Q}_p}$ and $f_{F}:=f_{F/\mbb{Q}_p}$, 
and  also set $q_F:=p^{f_F}$. 
If $F$ is a $p$-adic field, we denote by 
$F^{\mrm{ab}}$ and $F^{\mrm{ur}}$ the maximal abelian extension of $F$ 
and the maximal unramified extension of $F$, respectively.


\section{Proof}

\subsection{Some technical tools}
We denote by $v_p$ the $p$-adic valuation on 
a fixed algebraic closure $\overline{\mbb{Q}}_p$ of $\mbb{Q}_p$
normalized by $v_p(p)=1$. 
Let $K$ be a $p$-adic field.
For  any  continuous character 
$\chi$ of $G_{K}$, we often regard $\chi$ as a character of $\mrm{Gal}(K^{\mrm{ab}}/K)$.
We denote by $\mrm{Art}_K$  
the local Artin map $K^{\times}\to \mrm{Gal}(K^{\mrm{ab}}/K)$
with arithmetic normalization.
We set
$\chi_{K}:=\chi\circ \mrm{Art}_{K}$. 
We denote by $\widehat{K}^{\times}$ 
the profinite completion of $K^{\times}$.
Note that  the local Artin map induces a topological isomorphism
$\mrm{Art}_K\colon \widehat{K}^{\times}\overset{\sim}{\rightarrow} \mrm{Gal}(K^{\mrm{ab}}/K)$.
\begin{proposition}
\label{vplem} 
Let $K$ and $k$ be $p$-adic fields.
We denote by $k_{\pi}$ the Lubin-Tate extension of $k$ associated with 
a uniformizer $\pi$  of $k$. 
$($If $k=\mbb{Q}_p$ and $\pi=p$, then we have $k_{\pi}=\mbb{Q}_p(\mu_{p^{\infty}})$.$)$
Let
$\chi_1,\dots , \chi_n \colon G_{K}\to \overline{\mbb{Q}}_p^{\times}$  be continuous characters.
Then we have 
\begin{align*}
& \ \mrm{Min}\left\{ \sum^n_{i=1} v_p(\chi_i(\sigma)-1) \mid \sigma\in G_{Kk_{\pi}} \right\} 
\\
\le & \   
\mrm{Min}\left\{ \sum^n_{i=1} v_p(\chi_{i,K}\circ \mrm{Nr}_{Kk/K}(\omega)-1) 
\mid \omega\in \mrm{Nr}_{Kk/k}^{-1}(\pi^{f_{Kk/k}\mbb{Z}}) \right\}.
\end{align*}
\end{proposition}

\begin{proof}
We have a topological isomorphism
$\mrm{Art}^{-1}_k\colon  \mrm{Gal}(k^{\mrm{ab}}/k) \overset{\sim}{\rightarrow} 
\widehat{k}^{\times}$ and $\mrm{Art}^{-1}_k(\mrm{Gal}(k^{\mrm{ab}}/k^{\mrm{ur}}))=\cO_k^{\times}$.
We denote by $M$ the maximal unramified extension of $k$ contained in $Kk$.
Since the group $\mrm{Art}^{-1}_k(\mrm{Gal}(k^{\mrm{ab}}/M))$ contains $\cO_k^{\times}$
and is a subgroup of $\widehat{k}^{\times}=\pi^{\widehat{\mbb{Z}}}\times \cO_k^{\times}$
of index $[M:k]$, we see 
$\mrm{Art}^{-1}_k(\mrm{Gal}(k^{\mrm{ab}}/M))
=\pi^{[M:k]\widehat{\mbb{Z}}}\times \cO_k^{\times}$.
On the other hand,  
we have $\mrm{Art}^{-1}_k(\mrm{Gal}(k^{\mrm{ab}}/k_{\pi}))=\pi^{\widehat{\mbb{Z}}}$.
Thus we obtain $\mrm{Art}^{-1}_k(\mrm{Gal}(k^{\mrm{ab}}/Mk_{\pi}))=\pi^{[M:k]\widehat{\mbb{Z}}}$.
Now we denote by $\mrm{Res}_{Kk/k}$ the natural restriction map 
$\mrm{Gal}((Kk)^{\mrm{ab}}/Kk)\to \mrm{Gal}(k^{\mrm{ab}}/k)$.
It is not difficult to check that 
$\mrm{Res}_{Kk/k}^{-1}(\mrm{Gal}(k^{\mrm{ab}}/Mk_{\pi}))
=\mrm{Gal}((Kk)^{\mrm{ab}}/Kk_{\pi})$.
Thus it follows that the group
$\mrm{Art}^{-1}_{Kk}(\mrm{Gal}((Kk)^{\mrm{ab}}/Kk_{\pi}))$
coincides with $\mrm{Nr}_{Kk/k}^{-1}(\pi^{[M:k]\widehat{\mbb{Z}}})$.
Therefore, if we take any $\omega\in \mrm{Nr}_{Kk/k}^{-1}(\pi^{[M:k]\mbb{Z}})$, 
we have 
\begin{align*}
& \ \mrm{Min}\left\{ \sum^n_{i=1} v_p(\chi_i(\sigma)-1) \mid \sigma\in G_{Kk_{\pi}} \right\}  \\
= & \ 
\mrm{Min}\left\{ \sum^n_{i=1} v_p(\chi_i(\sigma)-1) \mid \sigma\in \mrm{Gal}((Kk)^{\mrm{ab}}/Kk_{\pi}) \right\} \\
= & \
\mrm{Min}\left\{ \sum^n_{i=1} v_p(\chi_{i, K}\circ \mrm{Nr}_{Kk/K}\circ \mrm{Art}^{-1}_{Kk}(\sigma)-1) 
\mid \sigma\in \mrm{Gal}((Kk)^{\mrm{ab}}/Kk_{\pi}) \right\}  \\
\le & \
\sum^n_{i=1} v_p(\chi_{i,K}\circ \mrm{Nr}_{Kk/K}(\omega)-1).
\end{align*}
\end{proof} 

We recall an observation of Conrad.
We denote by $\underline{K}^{\times}$ 
the Weil restriction  $\mrm{Res}_{K/\mbb{Q}_p}(\mbb{G}_m)$
and let $D^K_{\mrm{cris}}(\cdot) := (B_{\mrm{cris}}\otimes_{\mbb{Q}_p} \cdot)^{G_K}$.
\begin{proposition}[{\cite[Proposition B.4]{Co}}]
\label{LA2}
Let $K$ and $F$ be $p$-adic fields.
Let $\chi\colon G_K\to F^{\times}$ be 
a continuous character. 
We denote by $F(\chi)$ the $\mbb{Q}_p$-representation of $G_K$
underlying a $1$-dimensional $F$-vector space endowed with an $F$-linear action by $G_K$ via $\chi$,

\noindent
{\rm (1)} $\chi$ is crystalline\footnote{This means that 
the $\mbb{Q}_p$-representation $F(\chi)$ of $G_K$ is crystalline.}
 if and  only if 
there exists a {\rm(}necessarily unique{\rm )} $\mbb{Q}_p$-homomorphism 
$\chi_{\rm alg}\colon \underline{K}^{\times}\to \underline{F}^{\times}$
such that $\chi_K$ and $\chi_{\rm alg}$ {\rm (}on $\mbb{Q}_p$-points{\rm )} 
coincides on $\cO_K^{\times}\ (\subset K^{\times}=\underline{K}^{\times}(\mbb{Q}_p))$.

\noindent
{\rm (2)} Let $K_0$ be the maximal unramified subextension 
of $K/\mbb{Q}_p$. 
Assume that $\chi$ is crystalline and let $\chi_{\rm alg}$ be as in {\rm (1)}. 
{\rm (}Note that $\chi^{-1}$ is also  crystalline.{\rm )}
Then, the filtered $\vphi$-module 
$D^K_{\mrm{cris}}(F(\chi^{-1}))=(B_{\mrm{cris}}\otimes_{\mbb{Q}_p} F(\chi^{-1}))^{G_K}$
over $K$ is free of rank $1$ over $K_0\otimes_{\mbb{Q}_p} F$ 
and its $k_0$-linear endomorphism $\vphi^{f_K}$ 
is given by the action of the product 
$\chi_K(\pi_K)\cdot \chi_{\rm alg}^{-1}(\pi_K)\in F^{\times}$.  
Here, $\pi_K$ is any uniformizer of $K$.
\end{proposition}

We define some notations for later use. 
Assume that $K$ is a Galois extension of $\mbb{Q}_p$.
Let $\chi\colon G_K\to K^{\times}$ be a crystalline character. 
Let $\chi_{\mrm{LT}} \colon I_{ K}\to K^{\times}$ 
be the restriction to the inertia $ I_K$ of the Lubin-Tate character associated  
with any choice of uniformizer of $K$ 
(it depends on the choice of a uniformizer of $K$,
but its restriction to the inertia subgroup does not).
By definition, the character $\chi_{\mrm{LT}}$ is characterlized by 
$\chi_{\mrm{LT}}\circ \mrm{Art}_K(x)=x^{-1}$ for any $x\in \cO_K^{\times}$.
(We remark that $\chi_{\mrm{LT}}$ is the restriction to $I_K$ of the 
$p$-adic cyclotomic character  if $K=\mbb{Q}_p$.)
Then, we have 
$$
\chi = 
\prod_{\sigma \in \Gamma_K} \sigma^{-1} \circ \chi_{\mrm{LT}}^{h_{\sigma}}
$$
on the inertia $I_K$ for some (unique) integer $h_{\sigma}$. 
Equivalently,  the character
$\chi_{\rm alg}$ (appeared in Proposition \ref{LA2}) 
on $\mbb{Q}_p$-points is given by 
$$
\chi_{\rm alg}(x)= 
\prod_{\sigma \in \Gamma_K} (\sigma^{-1}x)^{-h_{\sigma}}
$$
for $x\in K^{\times}$. 
We say that $\mbf{h}=(h_{\sigma})_{\sigma\in \Gamma_K}$ 
is the {\it Hodge-Tate type of $\chi$}.
Note that $\{h_{\sigma} \mid \sigma\in \Gamma_K \}$ as a set  
is the set of Hodge-Tate weights of $K(\chi)$, 
that is, $C\otimes_{\mbb{Q}_p} K(\chi)
\simeq \oplus_{\sigma \in \Gamma_K} C(h_{\sigma})$
where $C$ is the completion of $\overline{\mbb{Q}}_p$.

For any set of integers $\mbf{h}=(h_{\sigma})_{\sigma\in \Gamma_K}$
indexed by $\Gamma_K$,
we define a continuous character 
$\psi_{\mbf{h}}\colon \cO_K^{\times}\to \cO_K^{\times}$ by 
\begin{equation}
\label{psih}
\psi_{\mbf{h}}(x)=\prod_{\sigma\in \Gamma_K} (\sigma^{-1}x)^{-h_{\sigma}}.
\end{equation}
\begin{lemma}
\label{infinite}
For $1\le i\le r$, let $\mbf{h}_i=(h_{i,\sigma})_{\sigma\in \Gamma_K}$
be a set of integers.
For each $i$, assume that
\begin{itemize}
\item[{\rm (a)}] $\sum_{\sigma\in \Gamma_K} h_{i,\sigma}$ is not zero, and 
\item[{\rm (b)}] $h_{i,\sigma}\not=h_{i,\tau}$ for some $\sigma,\tau\in \Gamma_K$.
\end{itemize}
Then, 
there exists an element 
$\omega$ of $\ker \mrm{Nr}_{K/\mbb{Q}_p}$ 
such that $\psi_{\mbf{h}_1}(\omega),\dots ,\psi_{\mbf{h}_r}(\omega)$ are of infinite orders.
\end{lemma}
\begin{proof}
For any character $\chi$ on $\cO_K^{\times}$,
we denote by $\chi'$ the restriction of $\chi$ to $1+p^2\cO_K$.
To show the lemma, it suffices to show
\begin{equation}
\label{sub}
\ker \mrm{Nr}_{K/\mbb{Q}_p}'
\not \subset  \bigcup^r_{i=1} \ker \psi'_{\mbf{h}_i}.
\end{equation}
(In fact, any non-trivial element of $\mrm{Im}\ \psi'_{\mbf{h}_i}$ is of infinite order
since $\mrm{Im}\ \psi'_{\mbf{h}_i}$ is a subgroup of 
a torsion free group $1+p^2\cO_K$.)
Since $N'_{K/\mbb{Q}_p}(1+p^2\cO_K)$ 
is an open subgroup of $\mbb{Z}_p^{\times}$,
we see that the dimension\footnote{If a profinite group $G$ has an open subgroup $U$ 
which is isomorphic to $\mbb{Z}_p^{\oplus d}$, then $d$ does not depend on the choice 
of $U$ and we say that $d$ is the {\it dimension of $G$}.
For example, $\dim \mbb{Z}_p^{\oplus d}=d$.
Note that the dimension of $G$ is zero if and only if $G$ is finite.
See \cite{DDMS} 
for general theories of  dimensions of $p$-adic analytic groups.} 
of $\ker N'_{K/\mbb{Q}_p}$ is $d_K-1$.
We claim that $\dim \ker \psi_{\mbf{h}_i}<d_K-1$.\ 
By the assumption (a),
we see that $\mrm{Im}\ \psi'_{\mbf{h}_i}$ contains 
an open subgroup $H$ of $\mbb{Z}_p^{\times}$. 
Thus we have 
$\dim \ker \psi'_{\mbf{h}_i}=d_K-\dim \mrm{Im}\ \psi'_{\mbf{h}_i}
\le d_K-1$.
If we assume $\dim \ker \psi_{\mbf{h}_i}'=d_K-1$, then 
$\dim \mrm{Im}\ \psi'_{\mbf{h}_i}=1$ and thus 
$H$ is a finite index subgroup of $\mrm{Im}\ \psi'_{\mbf{h}_i}$.
It follows that 
there exists an open subgroup $U$ of $\cO_K^{\times}$
such that $\psi_{\mbf{h}_i}$ restricted to $U$ 
has values in $\mbb{Z}_p^{\times}$.
By \cite[Lemma 2.4]{Oz1}, we obtain that 
$h_{i,\sigma}=h_{i,\tau}$ for any $\sigma,\tau\in \Gamma_K$
but this contradicts the assumption (b) in the statement of the lemma. 
Thus we conclude that $\dim \ker \psi_{\mbf{h}_i}'<d_K-1$.

Now we fix an isomorphism $\iota\colon 1+p^2\cO_K\simeq \mbb{Z}_p^{\oplus d_K}$ 
of topological groups. 
We define vector subspaces $N$ and $P_i$ of $\mbb{Q}_p^{\oplus d_K}$   by 
$N:=\iota(\ker \mrm{Nr}_{K/\mbb{Q}_p}')
\otimes_{\mbb{Z}_p}\mbb{Q}_p$ 
and 
$P_i:=\iota(\ker \psi_{\mbf{h}_i}')
\otimes_{\mbb{Z}_p}\mbb{Q}_p$.
We know that 
$\dim_{\mbb{Q}_p} N=d_K-1$
and 
$\dim_{\mbb{Q}_p} P_i<d_K-1$.
Assume that \eqref{sub} does not hold, that is,
$\ker \mrm{Nr}_{K/\mbb{Q}_p}' 
\subset  \bigcup^r_{i=1} \ker \psi'_{\mbf{h}_i}$.
Then we have 
$N  \subset  \bigcup^r_{i=1} P_i$. 
This implies 
$N  = \bigcup^r_{i=1} (N\cap P_i)$.
By the lemma below, 
we find that $N  = N\cap P_i\subset P_i$ for some $i$
but this contradicts the fact that $\dim_{\mbb{Q}_p} N>\dim_{\mbb{Q}_p} P_i$.
\end{proof}

\begin{lemma}
Let $V$ be a vector space over a field $F$ of characteristic zero.
Let $W_1,\dots ,W_r$ be vector subspaces of $V$.
If $V=\bigcup^r_{i=1} W_i$, then $V=W_i$ for some $i$.
\end{lemma}
\begin{proof}
We show by induction on $r$. 
The cases $r=1,2$ are clear.
Assume that the lemma holds for $r$ and suppose 
 $V=\bigcup^{r+1}_{i=1} W_i$.
We assume both $W_1\not \subset \bigcup^{r+1}_{i=2} W_i$
and $W_{r+1}\not \subset \bigcup^{r}_{i=1} W_i$ holds.
Then there exist elements 
$\mbf{x}_1\in  W_1\smallsetminus  \bigcup^{r+1}_{i=2} W_i$
and
$\mbf{x}_{r+1}\in  W_{r+1}\smallsetminus \bigcup^{r}_{i=1} W_i$.
It is not difficult  to check  that
we have $\lambda \mbf{x}_1+\mbf{x}_{r+1}
\notin  W_1\bigcup W_{r+1}$ for any $\lambda\in F^{\times}$.
Hence there exists an integer $2\le j_n\le r$ for each integer $n>0$
such that $n\mbf{x}_1+\mbf{x}_{r+1} \in  W_{j_n}$.
Take any integers $0<\ell<k$ so that $j_{\ell}=j_k(=:j)$. 
Then $(k-\ell)\mbf{x}_1=(k\mbf{x}_1+\mbf{x}_{r+1})-(\ell\mbf{x}_1+\mbf{x}_{r+1})\in W_j$.
Since $F$ is of characteristic zero,
we have $\mbf{x}_1\in W_j$ but this contradicts the fact that 
$\mbf{x}_1\notin  \bigcup^{r+1}_{i=2} W_i$.
Therefore, either 
$W_1 \subset \bigcup^{r+1}_{i=2} W_i$
or $W_{r+1}\subset \bigcup^{r}_{i=1} W_i$ holds.
This shows that 
$V=\bigcup^{r+1}_{i=2} W_i$ or $V=\bigcup^{r}_{i=1} W_i$
and the induction hypothesis implies $V=W_i$ for some $i$. 
\end{proof}

Finally we describe the following consequence of $p$-adic Hodge theory,
which is well-known for experts.
\begin{proposition}
\label{char}
Let $X$ be a proper smooth variety with good reduction over a $p$-adic field $K$.
Then we have 
$$
\mrm{det}(T-\vphi^{f_K}\mid D^{K}_{\mrm{cris}}(H^i_{\aet}(X_{\overline{K}},\mbb{Q}_p))
=\mrm{det}(T-\mrm{Frob}^{-1}_K\mid H^i_{\aet}(X_{\overline{K}},\mbb{Q}_{\ell}))
$$
for any prime $\ell \not =p$.  Here, 
$\mrm{Frob}_K$ stands for the arithmetic Frobenius of $K$.
\end{proposition}

\begin{proof}
Let $Y$ be the special fiber of a proper smooth model of $X$ over the integer ring of $K$.
By the crystalline conjecture shown by Faltings \cite{Fa} (cf. \cite{Ni}, \cite{Tsu}), 
we have an isomorphism 
$D^{K}_{\mrm{cris}}(H^i_{\aet}(X_{\overline{K}},\mbb{Q}_p)) 
\simeq K_0\otimes_{W(\mbb{F}_{q_K})} H^i_{\mrm{cris}}(Y/W(\mbb{F}_{q_K}))$
of $\vphi$-modules over $K_0$.  
It follows from Corollary 1.3 of \cite{CLS} 
(cf.\ \cite[Theorem 1]{KM} and \cite[Remark 2.2.4 (4)]{Na})
that the characteristic polynomial of 
$K_0\otimes_{W(\mbb{F}_{q_K})} H^i_{\mrm{cris}}(Y/W(\mbb{F}_{q_K}))$
for the ($f_K$-iterate) Frobenius action coincides with 
$\mrm{det}(T-\mrm{Frob}^{-1}_K\mid H^i_{\aet}(X_{\overline{K}},\mbb{Q}_{\ell}))$
for any prime $\ell \not =p$. 
Thus the result follows.
\end{proof}

\subsection{Proof of the main theorem}

Let $A$ be a $g$-dimensional abelian variety over $K$ with complex multiplication.
We denote by $L$ the field obtained by adjoining to $K$ all points of $A[12]$.
It follows from \cite[Theorem 4.1]{Si} that 
endomorphisms of $A$ are defined over $L$.
By the Raynaud's criterion of semistable reduction  \cite[Proposition 4.7]{Gr},
$A$  has semi-stable reduction over $L$. 
Moreover, $A$ has good reduction over $L$ 
since $A$ has complex multiplication \cite[Section 2, Corollary 1]{ST}. 
Since the extension degree of $L$ over $K$ is at most the order of $GL_{2g}(\mathbb{Z}/12\mathbb{Z})$
and there exist only finitely many $p$-adic field of a given degree, 
we immediately reduces a proof of  Theorem \ref{Main} to show the 
following

\begin{proposition}
\label{Main'}
There exists a  constant $\hat{C}(K,g)$, 
depending only on a $p$-adic field $K$ and an integer $g>0$, 
which satisfies the following property:
Let $A$ be a $g$-dimensional abelian variety over $K$ with the properties that 
$A$ has good reduction over $K$ and 
$\mrm{End}_K(A)\otimes_{\mbb{Z}} \mbb{Q}$ is a CM field of degree $2g$. Then
we have 
$$
\sharp A\left(K(\sqrt[p^{\infty}]{K})\right)_{\mrm{tors}}< \hat{C}(K,d).
$$
\end{proposition}

\begin{proof}
Since there exist  only finitely many $p$-adic field of a given degree, 
replacing $K$ by a finite extension, 
we may assume the following hypothesis:
\begin{itemize}
\item[(H)] $K$ is a Galois extension of $\mbb{Q}_p$ and $K$ contains all $p$-adic fields 
of degree $\le 2g$.
\end{itemize}
In the rest of  the proof, 
we set $M:=K(\sqrt[p^{\infty}]{K})$.
Let $A$ be a $g$-dimensional abelian variety over $K$ with the properties that 
$A$ has good reduction over $K$ and 
$F:=\mrm{End}_K(A)\otimes_{\mbb{Z}} \mbb{Q}$ is a CM field of degree $2g$.
Let $T=T_p(A):=\plim_{n} A[p^n]$ be the $p$-adic Tate module of $A$
and $V=V_p(A):=T_p(A)\otimes_{\mbb{Z}_p} \mbb{Q}_p$. 
Then $V$ is a free $F_p:=F\otimes \mbb{Q}_p$-module of rank 
one and the representation 
$\rho\colon G_K\to GL_{\mbb{Z}_p}(T)(\subset GL_{\mbb{Q}_p}(V))$
defined by the $G_K$-action on $T$ has values in 
$GL_{F_p}(V)=F_p^{\times}$.
In particular, $\rho$ is an abelian representation.
The representation
$V$ is a Hodge-Tate representation
with Hodge-Tate weights $0$ (multiplicity $g$) and $1$  (multiplicity $g$).
Moreover, 
$V$ is crystalline
since $A$ has good reduction over $K$. 
Fix an  isomorphism
$\iota \colon T\overset{\sim}{\rightarrow} \mbb{Z}_p^{\oplus 2g}$
of $\mbb{Z}_p$-modules. 
We have an isomorphism
$\hat{\iota}\colon GL_{\mbb{Z}_p}(T)\simeq GL_{2g}(\mbb{Z}_p)$ 
relative to $\iota$. 
We abuse notation by writing $\rho$
for the composite map 
$G_K\to GL_{\mbb{Z}_p}(T)
\simeq GL_{2g}(\mbb{Z}_p)$
of $\rho$ and $\hat{\iota}$.
Now let $P\in T$ and denote by $\bar{P}$ the image of $P$ in $T/p^nT$.
By definition, we have 
 $\iota(\sigma P)=\rho(\sigma)\iota(P)$
for $\sigma\in G_K$. 
Suppose that $\bar{P}\in(T/p^nT)^{G_M}$.
This implies  $\sigma P-P \in p^nT$ for any $\sigma\in {G_M}$. 
This is equivalent to say that 
 $(\rho(\sigma)-E)\iota(P) \in p^n\mbb{Z}_p^{\oplus 2g}$, and this in particular 
implies 
$\det (\rho(\sigma)-E)\iota(P) \in p^n\mbb{Z}_p^{\oplus 2g}$
for any $\sigma\in {G_M}$.
If we denote by $M_{\mrm{ab}}$ 
the maximal abelian extension of $K$ contained in $M$,
it holds that $\rho(G_M)=\rho(G_{M_{\mrm{{ab}}}})$
since $\rho(G_K)$ is abelian.
Thus we have 
\begin{equation}
\label{ram1}
\det (\rho(\sigma)-E)\iota(P) \in p^n\mbb{Z}_p^{\oplus 2g}
\quad \mbox{for any $\sigma\in {G_{M_{\mrm{ab}}}}$}.
\end{equation}
On the other hand, 
we set $G:=\mrm{Gal}(M/K)$ and $H:=\mrm{Gal}(M/K(\mu_{p^{\infty}}))$. 
Let $\chi_p\colon G_K\to \mbb{Z}_p^{\times}$ be the $p$-adic cyclotomic character.
Since we have $\sigma \tau \sigma^{-1}=\tau^{\chi_p(\sigma)}$ for any $\sigma\in G$
and $\tau\in H$,
we see $(G,G)\supset (G,H)\supset H^{\chi_p(\sigma)-1}$.
Hence we have a natural surjection
\begin{equation}
\label{H}
H/H^{\chi_p(\sigma)-1}
\twoheadrightarrow 
H/\overline{(G,G)}
=\mrm{Gal}(M_{\mrm{ab}}/K(\mu_{p^{\infty}}))
\quad \mbox{for any $\sigma\in G$}.
\end{equation}
Let $\nu$ be the smallest $p$-power integer with the properties that  $\nu>1$ and $\chi_p(G_K)\supset 1+\nu \mbb{Z}_p$.
Then \eqref{H} gives the fact  that 
$\mrm{Gal}(M_{\mrm{ab}}/K(\mu_{p^{\infty}}))$
is of exponent $\nu$, that is, 
$\sigma\in G_{K(\mu_{p^{\infty}})}$ implies 
$\sigma^{\nu}\in G_{M_{\mrm{ab}}}$.
Hence it follows from \eqref{ram1} 
that, for any  point 
$P\in T$ such that its image $\bar{P}$ in $T/p^nT$ is 
fixed by $G_M$, we have 
\begin{equation}
\label{ram1'}
\det (\rho(\sigma)^{\nu}-E)\iota(P) \in p^n\mbb{Z}_p^{\oplus 2g}
\quad \mbox{for any $\sigma\in {G_{K(\mu_{p^{\infty}})}}$}.
\end{equation}

\begin{claim}
\label{claim1}
There exists a constant $C_0(K,g)$, depending only on $K$ and $g$ 
such that 
$$
v_p(\det (\rho(\sigma_0)^{\nu}-E))\le C_0(K,g)
$$
for some $\sigma_0\in G_{K(\mu_{p^{\infty}})}$.
\end{claim}
Admitting this claim,  we can finish the proof of Proposition \ref{Main'}
immediately: 
It follows from Claim \ref{claim1}  and \eqref{ram1'} that 
$(T/p^nT)^{G_M}\subset  p^{n-C_0(K,g)}T/p^nT$ 
for $n>C_0(K,g)$.
Setting $C(K,g)_p:=p^{C_0(K,g)2g}$, we obtain 
$\sharp A(M)[p^n]
= \sharp (T/p^nT)^{G_M}
\le \sharp (T/p^{C_0(K,g)}T)
= C(K,g)_p$,
which shows 
$\sharp A(M)[p^{\infty}]
\le  C(K,g)_p$,
On the other hand,
we remark that Kubo and Taguchi showed in \cite[Lemma 2.3]{KT} that 
the residue field $\mbb{F}_{M}$ of $M$ is finite. 
The reduction map indues an injection 
from the prime-to-$p$ part of $A(M)$
into $\overline{A}(\mbb{F}_M)$
where $\overline{A}$ is the reduction of $A$.
If we denote by $q$ the order of $\mbb{F}_M$,
it follows from the Weil bound that 
$\sharp \overline{A}(\mbb{F}_M)\le (1+\sqrt{q})^{2g}$.
Therefore, setting $C(K,g):=C(K,g)_p \cdot (1+\sqrt{q})^{2g}$, 
we conclude that 
$\sharp A(M)_{\mrm{tors}}\le C(K,g)$. 
This finishes the proof of the proposition.

It suffices to show Claim \ref{claim1}.
Since the action of $G_K$ on $V$ 
factors through an abelian quotient of $G_K$,
it follows from the Schur's lemma that each Jordan H\"older factor of 
$V\otimes_{\mbb{Q}_p} \overline{\mbb{Q}}_p$ is of dimension one. 
Let $\psi_1,\dots ,\psi_{2g}\colon G_K\to \overline{\mbb{Q}}_p^{\times}$
be the characters associated with the Jordan H\"older factors of 
$V\otimes_{\mbb{Q}_p} \overline{\mbb{Q}}_p$. 
Since $K$ contains all $p$-adic fields 
of degree $\le 2g$, we know that each $\psi_i$ has values in $K^{\times}$
(in fact, for any $\sigma\in G_K$, we know that 
$\psi_1(\sigma),\dots ,\psi_{2g}(\sigma)$
are the roots of the polynomial $\mrm{det}(T-\sigma\mid V) \in \mbb{Q}_p[T]$
of degree $2g$). 
In the rest of the proof, 
we regard $\psi_i$ as a character $G_K\to K^{\times}$ of $G_K$ 
with values in  $K^{\times}$. 
We remark that each $\psi_i$ is a crystalline character
since $V$ is crystalline.
Furthermore, we have 
$$
v_p(\det (\rho(\sigma)^{\nu}-E))
=v_p\left(\prod^{2g}_{i=1}(\psi_i^{\nu}(\sigma)-1)\right)
=\sum^{2g}_{i=1}v_p(\psi_i^{\nu}(\sigma)-1)
$$
for any $\sigma\in G_{K(\mu_{p^{\infty}})}$.
Hence it follows from Lemma \ref{vplem} that 
we have 
\begin{align}
\label{ram2}
& \ \mrm{Min}\left\{ v_p(\det (\rho(\sigma)^{\nu}-E) \mid \sigma\in 
G_{K(\mu_{p^{\infty}})} \right\}  \notag
\\
\le & \   
\mrm{Min}\left\{ \sum^{2g}_{i=1} v_p(\psi_{i,K}^{\nu}(p\omega)^{-1}-1) 
\mid \omega \in \ker \mrm{Nr}_{K/\mbb{Q}_p} \right\}.
\end{align}
Note that we have 
\begin{align}
\label{vpcal}
\psi_{i,K}(p\omega)^{-1}
& =
\psi_{i,K}(\pi_K^{-e_K}\cdot \pi_K^{e_K} p^{-1})\cdot \psi_{i,K}(\omega)^{-1} 
\notag \\
& = 
\psi_{i,K}(\pi_K)^{-e_K} \psi_{i,\mrm{alg}}(\pi_K^{e_K} p^{-1})
\cdot \psi_{i,K}(\omega)^{-1} 
 \notag \\
& = 
\alpha_i^{-e_K} \cdot \psi_{i,\mrm{alg}}(p)^{-1} \cdot 
\psi_{i,K}(\omega)^{-1} 
\end{align}
for $\omega \in \ker \mrm{Nr}_{K/\mbb{Q}_p}$ 
where $\alpha_i:=\psi_{i,K}(\pi_K)\psi_{i,\mrm{alg}}(\pi_K)^{-1}$.
\begin{lemma}
\label{alpha}
Let the notation be as above.
Let $A^{\vee}$ be the dual abelian variety of $A$,
and let  $\overline{A}$ and $\overline{A^{\vee}}$ be the reductions of 
$A$ and $A^{\vee}$, respectively.

\noindent
{\rm (1)} $\alpha_i$ is a root of the characteristic polynomial 
of the geometric Frobenius endomorphism of $\overline{A}_{/\mbb{F}_K}$.

\noindent
{\rm (2)} $\alpha_i^{-1}q_K$ is a root of the characteristic polynomial 
of the geometric Frobenius endomorphism of  $\overline{A^{\vee}}_{/\mbb{F}_K}$.
\end{lemma}
\begin{proof}
Since $K(\psi_i^{-1})$ is a subquotient of $V_p(A)^{\vee}\otimes_{\mbb{Q}_p} K$,
it follows from  Proposition \ref{LA2} that  $\alpha_i$ is a root of the characteristic polynomial 
$f(T):=\mrm{det}(T-\varphi^{f_K}\mid D^K_{\mrm{cris}}(V_p(A)^{\vee}))$
of the $K_0$-linear endomorphism $\varphi^{f_K}$,
the $f_K$-th iterate of the Frobenius $\vphi$, 
on the $K_0$-vector space $D^K_{\mrm{cris}}(V_p(A)^{\vee})$.
We find  that 
\begin{align*}
f(T) 
& =\mrm{det}(T-\varphi^{f_K}\mid D^K_{\mrm{cris}}
(H^1_{\aet}(A_{\overline{K}},\mbb{Q}_p)) \\
& = \mrm{det}(T-\mrm{Frob}^{-1}_{K}\mid 
H^1_{\aet}(A_{\overline{K}},\mbb{Q}_{\ell}) 
= \mrm{det}(T-\mrm{Frob}_K\mid 
V_{\ell}(\overline{A}))
\end{align*}
for any prime $\ell\not=p$
where $\mrm{Frob}_K$ stands for the arithmetic Frobenius.
The second equality follows from Proposition \ref{char}.
The last term above coincides 
with  the characteristic polynomial of 
the geometric Frobenius endomorphism of $\overline{A}_{/\mbb{F}_K}$. This shows (1).
On the other hand,  it follows from Proposition \ref{LA2} again that 
$\alpha_i^{-1}$ is  a root of 
$\mrm{det}(T-\varphi^{f_K}\mid D^K_{\mrm{cris}}(V_p(A)))$.
Since $ V_p(A)(-1)\simeq V_p(A^{\vee})^{\vee}$, we see that 
$\alpha_i^{-1}q_K$ is a root of 
$f^{\vee}(T):=\mrm{det}(T-\varphi^{f_K}\mid D^K_{\mrm{cris}}(V_p(A^{\vee})^{\vee}))$.
Now the same argument of the proof of (1) with replacing $A$ by $A^{\vee}$
gives a proof of  (2).
\end{proof}

We continue the proof of Proposition \ref{Main'}.
Let 
$\mbf{h}_i=(h_{i,\sigma})_{\sigma\in \Gamma_K}$ 
be the Hodge-Tate type of $\psi_i$. 
Then we have  $h_{i,\sigma} \in \{0,1\}$ for any $i$ and $\sigma$.
We may suppose the following:
\begin{itemize}
\item[(I)] $\mbf{h}_i\not=(0)_{\sigma\in \Gamma_K},(1)_{\sigma\in \Gamma_K}$ 
for $1\le i\le r$, and 
\item[(II)] $\mbf{h}_i=(0)_{\sigma\in \Gamma_K}$ or $\mbf{h}_i=(1)_{\sigma\in \Gamma_K}$
for $r+1\le i\le 2g$.
\end{itemize}
Consider the case 
$\mbf{h}_i=(0)_{\sigma\in \Gamma_K}$.
If this is the case,  $\psi_i$ is unramified. This implies that 
$\psi_{i,\mrm{alg}}$ on ($\mbb{Q}_p$-points) is trivial. 
Take any $\omega \in \ker \mrm{Nr}_{K/\mbb{Q}_p}$
and consider the $p$-adic value $v_p(\psi_{i,K}^{\nu}(p\omega)^{-1}-1)$.
By \eqref{vpcal}, we have 
\begin{equation}
\label{HT0}
\psi_{i,K}^{\nu}(p\omega)^{-1}=\alpha_i^{-\nu e_K}.
\end{equation}
We remark that the right hand side is independent of the choice 
of $\omega \in \ker \mrm{Nr}_{K/\mbb{Q}_p}$
and $\alpha_i$ must be a $p$-adic unit (since so is the left hand side).
Next consider the case 
$\mbf{h}_i=(1)_{\sigma\in \Gamma_K}$. 
If this is the case, 
we have $\psi_i=\chi_p$ on $I_K$, 
that is, $\psi_{i,\mrm{alg}}$ (on $\mbb{Q}_p$-points) is $\mrm{Nr}_{K/\mbb{Q}_p}^{-1}$. 
Take any $\omega \in \ker \mrm{Nr}_{K/\mbb{Q}_p}$
and consider the $p$-adic value $v_p(\psi_{i,K}^{\nu}(p\omega)^{-1}-1)$.
By \eqref{vpcal}, we have 
\begin{equation}
\label{HT1}
\psi_{i,K}^{\nu}(p\omega)^{-1}=(\alpha_i^{-e_K}\cdot 
\mrm{Nr}_{K/\mbb{Q}_p}(p))^{\nu}=(\alpha_i^{-1}q_K)^{\nu e_K}.
\end{equation}
We remark that the last term is independent of the choice 
of $\omega \in \ker \mrm{Nr}_{K/\mbb{Q}_p}$.

Suppose $r+1\le i\le 2g$. 
Let $L$ be the unramified extension of $K$ of degree $\nu e_K$.  
Denote by $f_i(T)$ 
the characteristic polynomial of the Frobenius endomorphism of  
$\overline{A}_{/\mbb{F}_{L}}$ (resp.\ $\overline{A^{\vee}}_{/\mbb{F}_{L}}$) 
if $\mbf{h}_i=(0)_{\sigma\in \Gamma_K}$ 
(resp.\ $\mbf{h}_i=(1)_{\sigma\in \Gamma_K}$).
It follows from \eqref{HT0} (resp.\ \eqref{HT1}) and Lemma \ref{alpha} 
that $\psi_{i,K}^{\nu}(p\omega)$ (resp.\  $\psi_{i,K}^{\nu}(p\omega)^{-1}$) is a unit root of $f_i(T)$. 
Since $f_i(1)$ coincides with 
$\sharp \overline{A}(\mbb{F}_{q_{L}})$ 
(resp.\ $\sharp \overline{A^{\vee}}(\mbb{F}_{q_{L}})$), 
we find
$v_p(\psi_{i,K}^{\nu}(p\omega)^{-1}-1)\le v_p(f_i(1))$.
It follows from the Weil bound that 
$f_i(1)\le (1+\sqrt{q_L})^{2g}\le (1+\sqrt{p}^{\nu d_K})^{2g}$, which gives 
an inequality $v_p(f_i(1))\le \log_p (1+\sqrt{p}^{\nu d_K})^{2g}$. 
Therefore, setting $C_2(K,g):=\log_p (1+\sqrt{p}^{\nu d_K})^{2g}$, 
we obtain 
\begin{equation*}
v_p(\psi_{i,K}^{\nu}(p\omega)^{-1}-1)\le C_2(K,g) 
\end{equation*}
for $r+1\le i\le 2g$.

Suppose $1\le i\le r$. 
We define a subset $\mcal{R}=\mcal{R}(K,g)$ of $\overline{\mbb{Q}}_p$ 
by the set consisting of $\alpha\in \overline{\mbb{Q}}_p$
which is a root of a polynomial in $\mbb{Z}[T]$ of degree at most $2g$
and also is a $q_K$-Weil integer of weight $1$.
We also define 
$\mcal{R}'=\mcal{R}'(K,g):=
\{(\alpha^{-e_K}p^h)^{\nu} \mid \alpha\in \mcal{R}, 0< h< d_K \}$.
Then, both $\mcal{R}$ and  $\mcal{R}'$ are finite sets 
and depend only on $K$ and $g$.
Furthermore, Lemma \ref{alpha} and the Weil Conjecture imply 
that each $\alpha_i$ is
an element of $\mcal{R}$. 
Thus, setting 
$\gamma_i:=\alpha_i^{-e_K} \cdot \psi_{i,\mrm{alg}}(p)^{-1}
=\alpha_i^{-e_K} \cdot p^{\sum_{\sigma \in \Gamma_K}h_{i,\sigma}}$,
we have $\gamma_i^{\nu}\in \mcal{R}'$.
We consider the continuous character 
$\psi_{\mbf{h}_i}\colon \cO_K^{\times}\to \cO_K^{\times}$
defined in \eqref{psih}.
The character $\psi_{i,\mrm{alg}}$ (on $\mbb{Q}_p$-points) 
restricted to $\cO_K^{\times}$ coincides with $\psi_{\mbf{h}_i}$.
By Lemma \ref{infinite},  
there exists an element 
$\omega=\omega(K;\mbf{h}_1,\dots \mbf{h}_r)$ of $\ker \mrm{Nr}_{K/\mbb{Q}_p}$
such that $\psi_{\mbf{h}_1}^{\nu}(\omega),\dots ,\psi_{\mbf{h}_r}^{\nu}(\omega)$
 are of infinite order.
Since $\mcal{R}'$ is finite,
there exists 
an integer $r$
such that 
$\psi_{\mbf{h}_1}^{\nu}(\omega^r),\dots ,\psi_{\mbf{h}_r}^{\nu}(\omega^r)$
are not contained in  $\mcal{R}'$. 
Putting $\omega_0=\omega^r$, 
it holds that
\begin{itemize}
\item $\omega_0$ is an element of $\ker \mrm{Nr}_{K/\mbb{Q}_p}$.
Furthermore,  $\omega_0$ depends only on $K,g$ and $\mbf{h}_1,\dots ,\mbf{h}_r$, and  
\item $\psi_{\mbf{h}_1}^{\nu}(\omega_0),\dots ,\psi_{\mbf{h}_r}^{\nu}(\omega_0)$
are not contained in  $\mcal{R}'$. 
\end{itemize}
Now we define a constant
$C(K,g,\mbf{h}_1,\dots ,\mbf{h}_r)$
by 
$$
C(K,g,\mbf{h}_1,\dots ,\mbf{h}_r)
=\mrm{Max} \left\{ \sum^{r}_{i=1} v_p(\gamma_i'\psi_{\mbf{h}_i}^{\nu}(\omega_0)^{-1}-1) 
\mid \gamma_i'\in \mcal{R}'  \right\}.
$$
By construction of $\omega_0$, 
we see that the constant above is finite and depends only on $K,g,\mbf{h}_1,\dots ,\mbf{h}_r$.
We find that 
\begin{align}
\label{last}
& \mrm{Min}\left\{ \sum^{2g}_{i=1} v_p(\psi_{i,K}^{\nu}(p\omega)^{-1}-1) 
\mid \omega \in \ker \mrm{Nr}_{K/\mbb{Q}_p} \right\} \notag \\
\le &
\sum^{2g}_{i=1} v_p(\psi_{i,K}^{\nu}(p\omega_0)^{-1}-1) =
\sum^{r}_{i=1} v_p(\gamma_i^{\nu}\psi_{\mbf{h}_i}^{\nu}(\omega_0)^{-1}-1)  
+ \sum^{2g}_{i=r+1} v_p(\psi_{i,K}^{\nu}(p\omega_0)^{-1}-1) \notag \\
\le & C(K,g,\mbf{h}_1,\dots ,\mbf{h}_r)+(2g-r)C_2(K,g) \le C_0(K,g).
\end{align}
Here, 
\begin{equation*}
C_0(K,g):=\mrm{Max}\left\{  C(K,g,\mbf{h}_1,\dots ,\mbf{h}_r)+(2g-r)C_2(K,g) 
\mid 0\le r\le 2g,\  \mbf{h}_1,\dots ,\mbf{h}_r:\mbox{Case (I)} \right\}
\end{equation*}
(if $r=0$, we consider the constant $C(K,g,\mbf{h}_1,\dots ,\mbf{h}_r)$ as zero).
By construction, 
the constant $C_0(K,g)$ is finite and depends only on $K$ and $g$. 
By \eqref{ram2} and \eqref{last}, we conclude that $C_0(K,g)$ defined here 
satisfies the desired property 
of Claim \ref{claim1}. 
This is the end of the proof of Proposition \ref{Main'}.
\end{proof}

We end this paper with the following remarks.

\begin{remark}
(1) We do not know the explicit description of 
the bound $C(K,g)$ in Theorem \ref{Main}.  
 
\noindent
(2) We do not know whether we can remove the sentence "with complex multiplication"
from the statement of Theorem \ref{Main} or not.

\noindent
(3) Let $K$ be a $p$-adic field.  
Let $\pi=\pi_0$ be a uniformizer of $K$ and $\pi_n$ a $p^n$-th root of 
$\pi$ such that $\pi_{n+1}^p=\pi_n$ for any $n\ge 0$.
We set $K_{\infty}:=K(\pi_n\mid n\ge 0)$. 
The field $K_{\infty}$ is clearly a subfield of 
$K(\sqrt[p^{\infty}]{K})$.
It is well-known that  $K_{\infty}$ is one of key ingredients 
in (integral) $p$-adic Hodge theory
since $K_{\infty}$ is familiar to  the theory of norm fields. 
We can check the equality 
$$
A(K_{\infty})_{\mrm{tors}}=A(K)_{\mrm{tors}}
$$
holds for any abelain variety $A$ over $K$ with good reduction.
(We do not need CM assumption here.)
The proof is as follows: 
It  follows from the criterion of N\'eron-Ogg-Shafarevich \cite[Theorem 1]{ST} 
that the inertia subgroup $I_K$ of $G_K$ acts trivially 
on the prime-to-$p$ part  of $A(\overline{K})_{\mrm{tors}}$.
Since $K_{\infty}$ is totally ramified over $K$, 
we obtain the fact that the prime-to-$p$ parts of 
$A(K)_{\mrm{tors}}$ and $A(K_{\infty})_{\mrm{tors}}$ coincide with each other.
On the other hand,  we consider the following natural maps.
$$
A(K)[p^n]
\simeq 
\mrm{Hom}_{G_K}(\mbb{Z}/p^n\mbb{Z}, A(\overline{K})[p^n])
\overset{\iota}{\hookrightarrow} 
\mrm{Hom}_{G_{K_{\infty}}}(\mbb{Z}/p^n\mbb{Z}, A(\overline{K})[p^n])
\simeq 
A(K_{\infty})[p^n]
$$
Since $A$ has good reduction, the injection $\iota$ above is bijective
(cf.\ \cite[Theorem 3.4.3]{Br} for $p>2$; \cite{Ki}, \cite{La}, \cite{Li} for $p=2$).
This implies 
$
A(K_{\infty})[p^{\infty}]=A(K)[p^{\infty}].
$

\noindent
(4) It follows immediately from (3), the Raynaud's criterion of semistable reduction 
and  the main theorem of \cite{CX} 
that there exists an explicitly calculated constant $C$,
depending only on $K$ and $g$, such that we have 
$$
\sharp A(K_{\infty})_{\mrm{tors}}<C
$$
for any abelain variety $A$ over $K$ with potential good reduction.
(We do not need CM assumption here.)
We leave the readers to give the explicit description of $C$ above.
\end{remark}

\end{document}